\documentclass{article}
\usepackage[utf8]{inputenc}

\usepackage{amsmath}
\usepackage{amsthm}
\usepackage{amsfonts}
\usepackage{xcolor}
\usepackage{graphicx}
\usepackage{epstopdf} 
\usepackage[hidelinks]{hyperref}
\usepackage{thmtools, thm-restate} 

\newtheorem{theorem}{Theorem}
\newtheorem{lemma}[theorem]{Lemma}

\newtheorem{conjecture}[theorem]{Conjecture}
    
\theoremstyle{remark}
\newtheorem{example}[theorem]{Example}

\newtheorem{definition}[theorem]{Definition}

\newcommand{\C}{\mathcal C}
\newcommand{\D}{\mathcal D}
\newcommand{\R}{\mathbb R}
\newcommand{\F}{\mathbb F}
\newcommand{\U}{\mathcal U}
\DeclareMathOperator{\CF}{CF}
\DeclareMathOperator{\code}{code}
\newcommand*{\eqdef}{\stackrel{\mbox{\normalfont\tiny def}}{=}}  

\newcommand{\ox}{, } 
\newcommand*\samethanks[1][\value{footnote}]{\footnotemark[#1]}

\title{Recognizing and Realizing Inductively Pierced Codes}

\author{
Ryan Curry
\thanks{Department of Mathematics, Colby College, Waterville, ME 04901, USA.}
\and
R. Amzi Jeffs
\thanks{Department of Mathematical Sciences, Carnegie Mellon University, Pittsburgh, PA 15213, USA\@. Supported by the National Science Foundation through Award No. 2103206.}
\and
Nora Youngs
\thanks{Department of Mathematics, Colby College, Waterville, ME 04901, USA. Supported by the Clare Boothe Luce Program and the Luce Foundation.}
\and
Ziyu Zhao
\samethanks[1]
 }

\date{\today}

\begin{document}

\maketitle

\begin{abstract}
We prove algebraic and combinatorial characterizations of the class of inductively pierced codes, resolving a conjecture of Gross, Obatake\ox and Youngs. 
Starting from an algebraic invariant of a code called its canonical form, we explain how to compute a piercing order in polynomial time, if one exists.
Given a piercing order of a code, we explain how to construct a realization of the code using a well-formed collection of open balls, and classify the minimal dimension in which such a realization exists.
\end{abstract}

\section{Introduction}

Given a collection $\U = \{U_1,\ldots, U_n\}$ of convex sets in $\R^d$, one may combinatorially record how these sets intersect and cover one another by computing the \emph{code} of the collection $\U$, defined as 
\begin{align*}
\code(\U) &\eqdef \Big\{\sigma\subseteq [n] \, \Big\vert\, \bigcap_{i\in\sigma} U_i \setminus \bigcup_{j\in[n]\setminus \sigma} U_j \neq \emptyset \Big\}\\
& =\, \{\sigma\subseteq [n]\, \vert\,  \text{there exists $p\in \R^d$ with $p\in U_i$ if and only if $i\in\sigma$}\}. 
\end{align*}
Above, $[n]\eqdef \{1,2,\ldots, n\}$. 
Observe that $\code(\U)$ is a subset of the Boolean lattice $2^{[n]}$. 
The collection $\U$ is called a \emph{realization} of $\code(\U)$, and the region $\bigcap_{i\in\sigma} U_i \setminus \bigcup_{j\in[n]\setminus \sigma} U_j$ is called the \emph{atom} of $\sigma$ in the realization~$\U$. 
Elements of a code are called \emph{codewords}. 
With this terminology, the codewords of $\code(\U)$ are exactly the subsets of $[n]$ whose atoms are nonempty in the realization $\U$. 

In 2013 Curto, Itskov, Veliz-Cuba\ox and Youngs \cite{civy} initiated the study of \emph{open convex codes}, which are the codes that have realizations consisting of convex open sets. 
Their work was motivated by the neuroscientific study of place cells, and for this reason it is typical to refer to the indices in $[n]$ as \emph{neurons}, and think of $\R^d$ as a stimulus space in which the various $U_i$ are the place fields where each neuron is active. 
Codes are a granular tool for capturing the intersection and covering patterns of a collection of sets, and so the study of codes that have realizations consisting of convex sets is of purely mathematical interest as well. 
Indeed, work in this direction has led to novel theorems in discrete geometry, interesting constructive results\ox and rich families of extremal examples \cite{cgik, local_obstructions, non_local_obstructions, undecidability, embedding_phenomena, order_forcing, rigid_structures, non_monotone, all_embedding_vectors}.  

\begin{example}
Consider the code $\C = \{123, 45, 12, 1,2,4,5,\emptyset\}$ on five neurons. 
Here we are eliminating brackets and commas in codewords to make our notation more concise---for example $123$ represents $\{1,2,3\}$.
We will adopt this convention elsewhere as well. 
Figure \ref{fig:first-example} shows a realization of $\C$ using open balls in the plane. 
\begin{figure}
    \centering
    \includegraphics{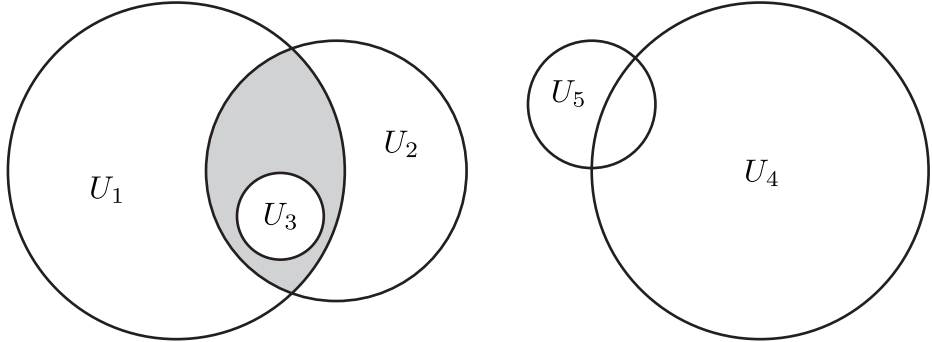}
    \caption{A realization of a code in the plane, with the atom of $12$ highlighted in gray.}
    \label{fig:first-example}
\end{figure}
\end{example}

Our work focuses exclusively on realizations consisting of convex open sets, and for this reason we use the more concise term \emph{convex code} in place of ``open convex code."
An efficient characterization of all convex codes is unfortunately out of the question; recent work by Lienkaemper, Kunin\ox and Rosen \cite{kunin_lienkaemper_rosen} indicates that recognizing whether or not a code is convex is NP-hard, even in $\R^2$. 
One way around this problem is to restrict to a simpler class of codes.
For example Cruz, Giusti, Itskov\ox and Kronholm \cite{cgik} gave a construction showing that every intersection complete code is convex.

In this vein, our work treats the class of \emph{$k$-inductively pierced} codes, which can be built up iteratively from certain intervals of codewords (see Definition \ref{def:inductively pierced} below).
We give an algebraic and geometric characterization of $k$-inductively pierced codes (Theorem \ref{thm:characterization}), proving a conjecture of Gross, Obatake\ox and Youngs \cite{visualization}. Furthermore, we exactly characterize the minimum dimension in which an inductively pierced code has a well-formed realization by open balls (Theorem \ref{thm:realization}).
Importantly, all of our characterizations are efficient in the following sense: given the canonical form of a code, we can compute a piercing order in polynomial time if one exists, and a piercing order allows us to quickly compute the smallest dimension in which a code can be realized by a well-formed collection of open balls.
This is discussed in Section \ref{sec:conclusion}. 
To present our results formally, we first require some additional background. 

\paragraph{Inductively Pierced Codes.}
For any $\sigma\subseteq \tau\subseteq [n]$, the \emph{interval} between $\sigma$ and $\tau$ is the set 
\[
[\sigma,\tau] \eqdef \{\gamma\subseteq [n]\mid \sigma\subseteq \gamma\subseteq \tau\}.
\]
The \emph{rank} of an interval $[\sigma,\tau]$ is equal to $|\tau\setminus \sigma|$. 
 Given a code $\C\subseteq 2^{[n]}$ and a neuron $i\in[n]$, we will often consider the code $\C\setminus i$, called the \emph{deletion} of $i$, 
which is obtained by deleting $i$ from every codeword in $\C$ where it appears.
For any $\sigma\subseteq [n]$, we let $\C\setminus \sigma$ denote the code obtained by successively deleting the neurons from~f$\sigma$.
Geometrically, if $\U = \{U_1,\ldots, U_n\}$ is a realization of $\C$ then deleting $U_i$ from this realization for all $i\in\sigma$ yields a realization of $\C\setminus \sigma$. 

Informally, a $k$-inductively pierced code is one which can be built up as a union of intervals with rank at most $k$ that fit together appropriately. 
\begin{definition}\label{def:k-piercing}\label{def:inductively pierced}
Let $\C\subseteq 2^{[n]}$ be a code and let $i\in[n]$ be a neuron.
We say that $i$ is an \emph{(abstract) $k$-piercing} of $\C$ if there exist $\sigma\subseteq \tau\subseteq [n]\setminus\{i\}$ so that
\begin{itemize}
    \item[(i)] $[\sigma,\tau]$ has rank $k$,
    \item[(ii)]$[\sigma,\tau]$ is contained in $\C\setminus i$, and 
    \item[(iii)] $\C = (\C\setminus i)\cup [\sigma\cup\{i\}, \tau\cup\{i\}]$.
\end{itemize}
We say that $[\sigma,\tau]$ is the interval \emph{associated to} the piercing. 
 A code $\C\subseteq 2^{[n]}$ is {\it $k$-inductively pierced} if $\C=\{\emptyset\}$, or there exists a neuron $i\in[n]$ which is a $k'$-piercing of $\C$ for some $k'\le k$, and $\C\setminus i$ is $k$-inductively pierced.
 A code is called simply \emph{inductively pierced} if it is $k$-inductively pierced for some value of $k$.
\end{definition}
 
From the definition, any $k$-inductively pierced code admits an ordering $i_1<i_2< \cdots < i_n$ of the neurons in~$[n]$ so that for each $j$, the neuron $i_j$ is a $k'$-piercing of $\C\setminus\{i_{j+1},\ldots,i_n\}$ for some $k'\le k$.
We call such an order a {\it $k$-piercing order}; note that there is no requirement that such an ordering be unique. 

Inductively pierced codes were previously studied by Gross, Obatake\ox and Youngs \cite{visualization}. 
Motivated by the study of Euler diagrams, their work focused on $0$-, $1$-\ox and $2$-inductively pierced codes and their possible realizations in the plane. 
We note that our definition of inductively pierced codes differs from theirs in notation, but is equivalent.
Our notation will allow us to more easily characterize inductively pierced codes, and describe their realizations in higher dimensions. 

\paragraph{The Neural Ideal.} 

Curto, Itskov, Veliz-Cuba\ox and Youngs \cite{civy} introduced an algebraic approach to understanding codes and their realizations.
Their approach used \emph{pseudo-monomials} in the polynomial ring $\F_2[x_1,\ldots, x_n]$, which are polynomials of the form $\prod_{i\in\sigma} x_i \prod_{j\in\tau} (1-x_j)$ where $\sigma$ and $\tau$ are disjoint subsets of $[n]$. 
It will be useful to partially order pseudo-monomials by divisibility, so that in particular we can pick out \emph{minimal} pseudo-monomials from a given set. 
Each $\sigma\subseteq [n]$ can be associated to an \emph{indicator pseudo-monomial}
\[
\rho_\sigma \eqdef \prod_{i\in\sigma} x_i \prod_{j\in[n]\setminus \sigma} (1-x_j).
\]
These indicators allow us to uniquely associate each code to an ideal in $\F_2[x_1,\ldots, x_n]$. 

\begin{definition}\label{def:neuralideal}
Given a code $\C\subseteq 2^{[n]}$, we define the {\it neural ideal} $J_\C\subseteq \F_2[x_1,\ldots, x_n]$ as follows: 
\[
J_\C \eqdef \langle \rho_\sigma \,|\, \sigma\notin\C\rangle.
\]
The {\it canonical form} of $J_\C$, denoted $\CF(J_\C)$, is the set of minimal pseudo-monomials in $J_\C$.
\end{definition}

Although the definition of $J_\C$ uses indicator pseudo-monomials, which have degree equal to $n$, the elements of the canonical form can in general have much lower degree. 
The pseudo-monomials in $J_\C$ have the following important geometric interpretation: $\prod_{i\in\sigma} x_i \prod_{j\in\tau} (1-x_j)$ is an element of $J_\C$ if and only if  $\bigcap_{i\in\sigma} U_i \subseteq \bigcup_{j\in \tau} U_j$ in every (possibly not convex or open) realization $\U = \{U_1,\ldots, U_n\}$ of $\C$. 
Thus $\CF(J_\C)$ can be regarded as a minimal description of the intersection and covering information in any realization of $\C$. 

Additionally, $J_\C$ provides a generalization of the well-studied Stanley-Reisner ideal \cite{miller_sturmfels, stanley} of a simplicial complex: if $\C$ is a simplicial complex then $\CF(J_\C)$ consists of monomials corresponding to minimal non-faces of $\C$, and $J_\C$ is the Stanely-Reisner ideal of $\C$ (see \cite[Section 4.4]{civy}). The ideal $J_\C$ and the canonical form are both unique to a particular code, and can be used to characterize codewords. Using a slight abuse of notation, we can evaluate polynomials in $J_\C$ on elements of $2^{[n]}$, where $f(\sigma)$ is computed by setting $x_i=1$ if $i\in \sigma$, and $x_i=0$ otherwise. We make use of the following useful facts: given $\sigma\subseteq [n]$, we have $\sigma\in \C$ if and only if $f(\sigma)=0$ for all $f\in J_\C$ and if and only if $f(\sigma)=0$ for all $f\in \CF(J_\C)$. 

\paragraph{Conventions.}
We are interested in using the canonical form to study inductively pierced codes and their realizations. 
To streamline our work, we will adopt several conventions, which we explain below.
\begin{itemize}
    \item[(1)] \textbf{Every code contains $\emptyset$ as a codeword.}
    This means that sets in a realization do not cover all of $\R^d$, and in particular we may restrict our attention to realizations consisting of bounded sets. 
    Algebraically, this means that $\prod_{i\in[n]}(1-x_i)$ is not in $J_\C$, and hence $\CF(J_\C)$ does not contain any pseudo-monomial that is entirely a product of $(1-x_i)$ terms. 
    \item[(2)] \textbf{Every neuron appears in some codeword.} 
    This means that no set in a realization is empty.
    Algebraically, this means that $x_i$ is not in $J_\C$ for any $i\in[n]$. 
    \item[(3)] \textbf{No two neurons appear in exactly the same codewords.}
    This means that no two sets in a realization are equal. 
    Note that if two neurons did have identical behavior, deleting one does not change whether or not a code is realizable, and so this assumption does not affect our analysis. 
    Algebraically, this means that for any $i\neq j$ we do not have both $x_i(1-x_j)$ and $x_j(1-x_i)$ in $J_\C$. 
\end{itemize}
Note that every code $\C\subseteq 2^{[n]}$ can be associated to a code $\D$ satisfying (2) and (3) above, so that $\D$ is convex if and only if $\C$ is convex.
Indeed, up to permutation, any neurons that do not appear can be forgotten, and if two neurons have identical behavior one of them can be discarded. 

\paragraph{Characterizing Inductively Pierced Codes.} 

Note that our conventions (1) and (2) imply that every pseudo-monomial in $\CF(J_\C)$ has degree at least two. 
We say that a code $\C$ is \emph{degree two} if every pseudo-monomial in $\CF(J_\C)$ has degree exactly two. 
Gross, Obatake\ox and Youngs \cite{visualization} showed that under the conventions (1)--(3), every $2$-inductively pierced code is degree two. 
Their techniques imply the more general result that any inductively pierced code is degree two. 
We will postpone the proof of this and our other results until Section \ref{sec:characterization}. 

\begin{restatable}{proposition}{degreetwoprop}\label{prop:CFdegreetwo} If $\C\subseteq 2^{[n]}$ is inductively pierced, then $\C$ is degree two---that is, all elements of $\CF(J_\C)$ have degree two. 
\end{restatable}

Following \cite{visualization}, we associate each degree two code $\C\subseteq 2^{[n]}$ to a graph $G(\C)$ called the \emph{general relationship graph} of $\C$, which has vertex set $[n]$ and an edge $ij$ whenever $\CF(J_\C)$ does not contain any pseudo-monomial whose two variables are $x_i$ and $x_j$.
Geometrically, $ij$ is an edge in $G(\C)$ if and only if $U_i$ and $U_j$ intersect and neither contains the other in every realization of $\C$. 
It turns out that the general relationship graph determines whether or not $\C$ is inductively pierced. 

Below, recall that a graph is \emph{chordal} if it can be built up by successively adding \emph{simplicial vertices}, whose neighborhood among previously added vertices is a clique. The reverse of the order in which these vertices are added is called a \emph{perfect elimination order}. 

\begin{restatable}{theorem}{characterizationthm}\label{thm:characterization}
A code $\C\subseteq 2^{[n]}$ is inductively pierced if and only if $\C$ is degree two and $G(\C)$ is chordal. 
If $k$ is the smallest integer so that $\C$ is $k$-inductively pierced, then the largest clique in $G(\C)$ has size $k+1$.
\end{restatable}

\paragraph{Computing Piercing Orders.} 
Theorem \ref{thm:characterization} gives a positive resolution to a conjecture of Gross, Obatake\ox and Youngs \cite{visualization}.
However, we would like to not only characterize inductively pierced codes, but also find a way to recognize them efficiently. 
To this end, we introduce a second auxiliary object associated to a degree two code.
Given a degree two code $\C$, let $P(\C)$ be the  partially ordered set whose elements are $[n]$ and which has comparability relation defined by 
\[
\text{$i < j$ in $P(\C)$ if and only if $x_i(1-x_j)\in \CF(J_\C)$}.
\]
Recall that $\CF(J_\C)$ cannot contain both $x_i(1-x_j)$ and $x_j(1-x_i)$ by convention (3), so this relation is antisymmetric. 
Moreover, if we have $i < j < k$ in $P(\C)$, then $x_i(1-x_j)$ and $x_j(1-x_k)$ both lie in $J_\C$, and so $x_i(1-x_k) = \big(x_i(1-x_j)\big)(1-x_k) + x_i\big(x_j(1-x_k)\big)$ lies in $J_\C$. 
In fact, $x_i(1-x_k)$ lies in $\CF(J_\C)$ since $\C$ is degree two, hence this relation is transitive, and defines a partial order.
As previously noted in the discussion on the neural ideal, we have $x_i(1-x_j)$ in $J_\C$ if and only if $U_i\subseteq U_j$ in every realization of $\C$, so $P(\C)$ can be regarded as capturing pairwise containments of sets in any realization of $\C$. 

We will use $P(\C)$ and $G(\C)$ to identify piercings of $\C$, and thus find piercing orders. 
It turns out that piercings of $\C$ correspond to elimination neurons, defined below. 
\begin{definition}\label{def:elimination-neuron}
Let $\C\subseteq 2^{[n]}$ be a degree two code.
We say that a neuron $i\in [n]$ is an \emph{elimination neuron} if $i$ is a simplicial vertex in $G(\C)$ and a minimal element in $P(\C)$. 
\end{definition}

\begin{restatable}{proposition}{elimneuronprop}\label{prop:elimination-neuron}
Let $\C\subseteq 2^{[n]}$ be a degree two code. 
Then a neuron $i\in [n]$ is a $k$-piercing of $\C$ if and only if $i$ is an elimination neuron with $k$ neighbors in $G(\C)$. 
\end{restatable}

  Note that Proposition \ref{prop:elimination-neuron} already provides a novel characterization of piercing orders: they are the linear orders on $[n]$ whose reverse is simultaneously a perfect elimination order of $G(\C)$ and a linear extension of $P(\C)$. 

\begin{restatable}{proposition}{greedyprop}\label{prop:greedy}
Let $\C\subseteq 2^{[n]}$ be a code, and suppose that a neuron $i\in[n]$ is a piercing of $\C$. Then $\C$ is inductively pierced if and only if $\C\setminus i$ is inductively pierced. 
\end{restatable}

Propositions \ref{prop:elimination-neuron} and \ref{prop:greedy} will let us argue that not only can we find piercings efficiently, but also that we will also never ``get stuck" when trying to build a piercing order.
In other words, a piercing order may be computed greedily.
For a full discussion of this algorithm, see Section \ref{sec:conclusion}. 
For now, we summarize our realization results. 

\paragraph{Realizing Inductively Pierced Codes.} 
Since inductively pierced codes have a significant amount of combinatorial structure, it is reasonable to hope that they admit structured realizations.
Rather than seek realizations by arbitrary convex open sets, we seek realizations by open balls. 
This can be achieved from any piercing order, and we can easily characterize the dimension in which such realizations exist. 
Below we make the additional requirement that our realizations are well-formed, in the sense that the boundaries of sets intersect generically.

\begin{definition}\label{def:well-formed}
A collection $\mathcal S = \{S_1, S_2, \ldots, S_n\}$ of $(d-1)$-dimensional spheres in $\R^d$ is called \emph{well-formed} if for every $k\le d$ the intersection of $k$-many spheres from $\mathcal S$ is either empty, or a $(d-k)$-dimensional sphere, and the intersection of $(d+1)$-many spheres from $\mathcal S$ is empty. A collection of open balls is called well-formed if their boundaries comprise a well-formed collection of spheres. 
\end{definition}

The above definition of well-formed spheres mirrors the definition of well-formed curves used by Gross, Obatake\ox and Youngs \cite{visualization} and the literature on Euler diagrams. 
Indeed, a collection of well-formed circles in $\R^2$ is a special case of a collection of well-formed curves. 
To characterize the exact dimensions in which an inductively pierced code admits a well-formed realization by open balls, we require two combinatorial definitions. 

\begin{definition}\label{def:attaching-set}
Let $\C\subseteq 2^{[n]}$ be a degree two code, and suppose that the largest clique in $G(\C)$ has size $k+1$. 
For every $\sigma\subseteq [n]$ of size $k$ on which $G(\C)$ is a clique, the \emph{attaching set} of $\sigma$ consists of all $i\in[n]\setminus\sigma$ whose neighborhood contains $\sigma$.
\end{definition}

In other words, the attaching set of $\sigma$ consists of the vertices which can be added to $\sigma$ to form a clique of size $k+1$. 
Note that the attaching set of $\sigma$ is always an independent set, as an edge between two of its elements would yield a clique of size $k+2$ in $G(\C)$. 

\begin{definition}\label{def:splittable}
Let $\C\subseteq 2^{[n]}$ be a degree two code, and suppose that the largest clique in $G(\C)$ has size $k+1$. 
We say that $\C$ is \emph{splittable} if the following holds for every $\sigma\subseteq [n]$ of size $k$ on which $G(\C)$ is a clique: the attaching set of $\sigma$ can be partitioned as $A\sqcup B$ so that $P(\C)$ restricted to the attaching set consists of two disjoint total orders on $A$ and $B$ respectively.
\end{definition}

\begin{restatable}{theorem}{realizationthm}\label{thm:realization}
Let $\C\subseteq 2^{[n]}$ be an inductively pierced code, and let $k\ge 1$. If the largest clique in $G(\C)$ has size $k+1$, then 
\begin{itemize}
    \item[(I)]$\C$ has a well-formed realization by open balls in $\R^{k+1}$,
    \item[(II)] $\C$ has such a realization in $\R^k$ if and only if $\C$ is splittable, and
    \item[(III)] $\C$ does not have such a realization in $\R^{k-1}$. 
\end{itemize}
\end{restatable}

Note that Theorem \ref{thm:realization} disregards 0-inductively pierced codes.
These codes are simple to realize since their realizations consist of sets that are either pairwise disjoint or nested, which can always be managed in $\R^1$. 

\begin{figure}
\[
\includegraphics{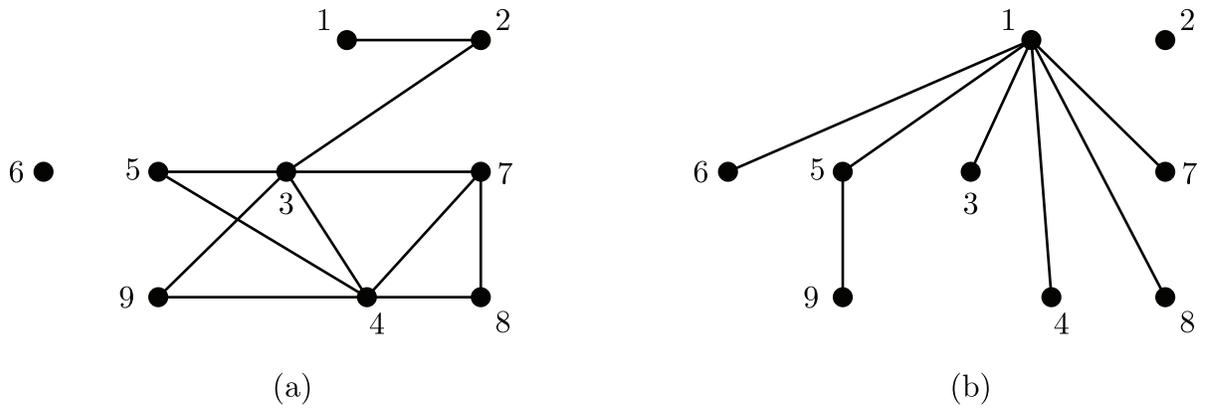}
\]
\caption{(a) The graph $G(\C)$ associated to $\C$. (b) The Hasse diagram of $P(\C)$.}
\label{fig:G-hat-C}
\end{figure}

\begin{figure}
\[
\includegraphics{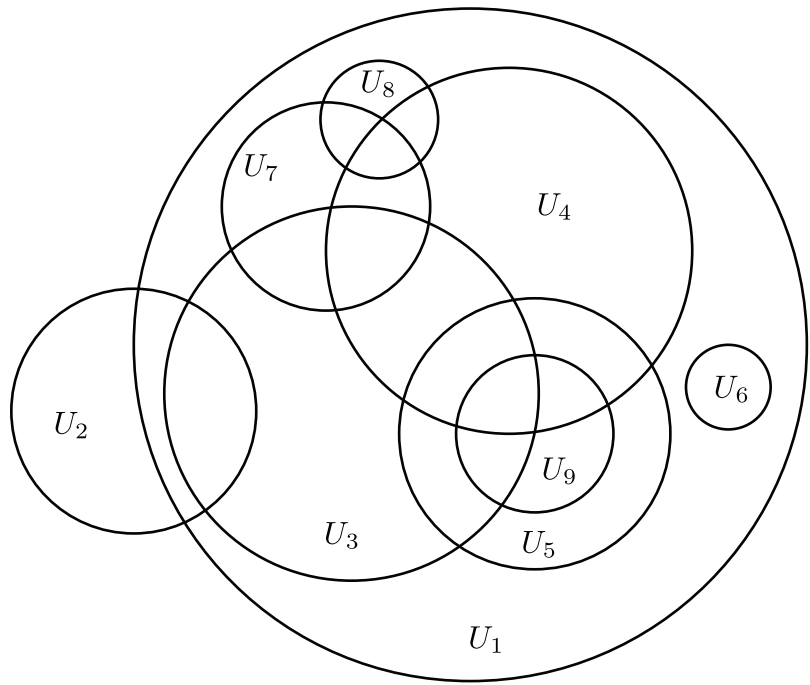}
\]
\caption{A well-formed realization of $\C$ by open balls in $\R^2$.}
\label{fig:2-dimensional-realization}
\end{figure}

\begin{example}
Consider the following code:
\begin{align*}
\C =&\{13459, 1347, 1478, 123, 16, \\
&\, \, \, 1345, 1359, 1459, 134, 135, 145, 137, 147, 148, 159, 178, 12, 13, 14, 15, 17, 18, 1, 2,\emptyset \}.
\end{align*}
Above, the maximal codewords appear on the first line and the remaining codewords are listed on the second line.
This code is 2-inductively pierced, and the natural linear order on $\{1,2,\ldots, 9\}$ is a piercing order.
It further turns out that $\C$ is splittable.
Note, for example, that the attaching set of $\{3,4\}$ is $\{5,7,9\}$, and the partition $\{5,9\}\sqcup \{7\}$ has $P(\C)$ a total order on each part. 
Figure \ref{fig:G-hat-C} shows $G(\C)$ and $P(\C)$. 
Item (II) of Theorem \ref{thm:realization} implies that $\C$ has a well-formed realization by open balls in $\R^2$, and Figure \ref{fig:2-dimensional-realization} shows such a realization. 
\end{example}

\begin{example}
Consider the following code:
\begin{align*}
\C =&\{123, 124, 125, \\
&\, \, \, 12, 13, 14, 15, 23, 24, 25, 1, 2, 3, 4, 5, \emptyset \}.
\end{align*}

This code is 2-inductively pierced; indeed, the natural linear order is a piercing order, and the largest clique in $G(\C)$ has size three. However, this code is not splittable. The attaching set of $\sigma=\{1,2\}$ is $\{3,4,5\}$ and this set cannot be partitioned into two linearly ordered subsets in $P(\C)$, since $P(\C)$ is an antichain on this set. Thus, Theorem \ref{thm:realization} implies that this code cannot be realized with open balls in $\R^2$, but such a realization is possible in $\R^3$.
One such realization in $\R^3$ consists of two intersecting unit balls $U_1$ and $U_2$, with $U_3$, $U_4$, and $U_5$ equal to small disjoint balls centered at points on the circle where the boundaries of $U_1$ and $U_2$ meet. 
\end{example}

\section{Characterizing Inductively Pierced Codes}\label{sec:characterization}

To begin, we justify several important statements about the canonical form, which may be of general interest.
First, we describe the effect of eliminating a neuron in any (not necessarily inductively pierced) code.
We then characterize how a piercing impacts the canonical form, again without assuming that the code in question is inductively pierced.
This yields a short proof of Proposition \ref{prop:CFdegreetwo}, establishing that every inductively pierced code is degree two.

\begin{lemma}\label{lem:deletion}
For any code $\C\subseteq 2^{[n]}$, the canonical form of $\C\setminus n \subseteq 2^{[n-1]}$ consists of exactly the pseudo-monomials in $\CF(J_\C)$ which do not depend on the variable $x_n$; that is,
\[
\CF(J_{\C\setminus n}) = \CF(J_\C) \setminus\{f\in \CF(J_\C) \mid x_n \text{ or } (1-x_n) \text{ divides }f\}.
\]
\end{lemma}

\begin{proof}
First we argue that the pseudo-monomials in $J_{\C\setminus n}$ are exactly those in $J_\C$ that do not depend on $x_n$. 
Let $f$ be any pseudo-monomial that does not depend on $x_n$.
Observe that $f(\sigma) = 0$ if and only if $f(\sigma\setminus \{n\}) = 0$ for every $\sigma\subseteq [n]$. 
Thus $f$ vanishes on every codeword of $\C$ if and only if $f$ vanishes on every codeword of $\C\setminus n$. 
Since the pseudo-monomials in the neural ideal of a code are exactly those which vanish on every codeword, we conclude that $f\in J_{\C\setminus n}$ if and only if $f\in J_\C$.
Consequently, the minimal pseudo-monomials in $J_{\C\setminus n}$ are the minimal pseudo-monomials in $J_\C$ that do not depend on $x_n$, proving the result.
\end{proof}

\begin{lemma}\label{lem:CF-pierced}
Let $\C\subseteq 2^{[n]}$ and let $n$ be a $k$-piercing of $\C$ with associated interval $[\sigma, \tau]$.
Then
\begin{equation}
\tag{$\star$}
\label{eq:CF-pierced}
 \CF(J_\C) =  \CF(J_{\C\setminus n})\sqcup \{x_n x_i \mid i\in [n-1]\setminus \tau\} \sqcup \{x_n(1-x_j) \mid j\in \sigma\}.
\end{equation}
\end{lemma}
\begin{proof}
By definition of $k$-piercing, the codewords of $\C$ containing $n$ form an interval $[\sigma\cup\{n\},\tau\cup\{n\}]$.
Hence $n$ only appears in codewords that contain $\sigma$, and contain no $i\in [n]\setminus \tau$. 
Consequently, $x_n(1-x_i)\in J_\C$ for all $i\in \sigma$ and $x_nx_i\in J_\C$ for every $i\notin \tau$.
Our conventions (1)--(3) forbid $J_\C$ from containing any degree one polynomials, and so these pseudo-monomials are minimal in $J_\C$, and hence lie in $\CF(J_\C)$. 
Lemma \ref{lem:deletion} shows that $\CF(J_{\C\setminus n})$ is obtained from $\CF(J_\C)$ by  deleting any pseudo-monomials in $\CF(J_\C)$ that depend on $x_n$.
Putting all of this together, we have
 \[
 \CF(J_\C) \supseteq \CF(J_{\C\setminus n})\sqcup \{x_nx_i \mid i\in [n-1]\setminus \tau\} \sqcup \{x_n(1-x_j) \mid j\in \sigma\}.
 \]
 
For the reverse containment, it will suffice to show that every pseudo-monomial in $J_\C$ is a multiple of one of the pseudo-monomials in the union of the three sets above.
Equivalently, we must show the following.
\smallskip

\noindent \underline{Claim:} If $f$ is a pseudo-monomial that is \emph{not} a multiple of some pseudo-monomial in 
\[
\CF(J_{\C\setminus n})\sqcup \{x_nx_i \mid i\in [n-1]\setminus \tau\} \sqcup \{x_n(1-x_j) \mid j\in \sigma\}\]
then $f(c)\neq 0$ for some codeword $c\in \C$.

\smallskip
\noindent \underline{Case 1: $f$ does not depend on $x_n$}.
Here Lemma \ref{lem:deletion} implies that $f$ does not lie in $J_{\C\setminus n}$. 
Hence there is a codeword $c'\in \C\setminus n$ with $f(c')\neq 0$. 
By the definition of a $k$-piercing, every $c'\in \C\setminus n$ is also a codeword of $\C$.\smallskip

\noindent \underline{Case 2: $(1-x_n)$ divides $f$}.
We may write $f = (1-x_n)g$ for some pseudo-monomial $g$ that does not depend on $x_n$. 
Since $f$ is not a multiple of any pseudo-monomial in $\CF(J_{\C\setminus n})$, we must have $g\notin J_{\C\setminus n}$.
Hence there is a codeword $c'\in \C\setminus n$ with $g(c') \neq 0$.
But by the definition of a $k$-piercing, $c'$ is also a codeword of $\C$, and furthermore we have $f(c') = g(c') \neq 0$ since $c'$ does not contain $n$.

\smallskip
\noindent \underline{Case 3: $x_n$ divides $f$}.
We can write $f = x_ng$ where $g$ is a pseudo-monomial that does not depend on $x_n$. 
Since $f$ is not a multiple of $x_nx_i$ for any $i\in [n-1]\setminus \tau$, we know that $x_i$ does not divide $g$ for any $i\in [n-1]\setminus \tau$. 
Similarly, since $f$ is not a multiple of any $x_n(1-x_j)$ for $j\in \sigma$, we know that $(1-x_j)$ does not divide $g$ for any $j\in \sigma$. 
Consequently, for some disjoint $\tau'\subseteq \tau$ and $\sigma'\subseteq [n-1]\setminus \sigma$ we have \[
g = \prod_{i\in \tau'} x_i \prod_{j\in \sigma'} (1-x_j).
\]
By the definition of a $k$-piercing, $\C$ contains the entire interval of codewords $[\sigma\cup\{n\}, \tau\cup \{n\}]$.
In particular, $c\eqdef \sigma \cup \tau' \cup\{n\}$ is a codeword of $\C$. 
We see directly that $f(c) \neq 0$, so the claim follows in this case as well, establishing the lemma. 
\end{proof}

\degreetwoprop*
\begin{proof}
We work by induction on $n$. 
As a base case, note that there is a unique code on one neuron that satisfies our conventions (1)--(3) in the introduction, namely $\C = \{1,\emptyset\}$. 
Here we have $\CF(J_\C)=\emptyset$, so $\C$ is degree two vacuously.
Now, assume that every inductively pierced code on $n-1$ neurons is degree two, and let $\C$ be an inductively pierced code on $n$ neurons. 
Without loss of generality, suppose $n$ is a $k$-piercing of $\C$.
By the formula (\ref{eq:CF-pierced}) in Lemma \ref{lem:CF-pierced}, the canonical form $\CF(J_\C)$ is obtained by adding degree two pseudo-monomials to $\CF(J_{\C\setminus n})$. 
The code $\C\setminus n$ is degree two by inductive hypothesis, and so the result follows.
\end{proof}

We now prove Proposition \ref{prop:elimination-neuron}, to establish that piercings correspond to elimination neurons.
As we will see, this implies that the graph $G(\C)$ associated to an inductively pierced code is chordal.
We are then able to prove Theorem \ref{thm:characterization}, which shows that, in fact, any degree two code whose corresponding graph is chordal must be inductively pierced.

Below we will work with degree two codes $\C$, and the associated graph $G(\C)$ and partial order $P(\C)$. 
For any neuron $i\in [n]$, we use the notation $N(i)$ for the neighborhood of $i$ in $G(\C)$, and $I(i)$ for the order ideal of $i$ in $P(\C)$.
That is, \begin{align*}
    N(i)\eqdef & \{j\in [n] \mid \text{ $i$ and $j$ adjacent in $G(\C)$}\}, \text{ and}\\
    I(i)\eqdef & \{j\in [n] \mid \text{ $j<i$ in $P(\C)$}\}. 
\end{align*}

\elimneuronprop*
\begin{proof}
Up to permutation, it suffices to establish that $n$ is a $k$-piercing of $\C$ if and only if $n$ is an elimination neuron with $k$ neighbors in $G(\C)$. 
First suppose that $n$ is a $k$-piercing of $\C$ with associated interval $[\sigma, \tau]$. 
By (\ref{eq:CF-pierced}) from Lemma \ref{lem:CF-pierced}, $(1-x_n)$ is not a factor of any pseudo-monomial in $\CF(J_\C)$, and so $n$ is minimal in $P(\C)$. 
The equation (\ref{eq:CF-pierced}) also implies that the neighborhood of $n$ in $G(\C)$ is exactly $\tau\setminus \sigma$, which has $k$ elements. 
It remains to argue that these neighbors form a clique.
Let $i,j\in \tau\setminus \sigma$, and note from the definition of a $k$-piercing that $\sigma$, $\sigma\cup\{i\}$, $\sigma\cup \{j\}$\ox and $\sigma\cup \{i,j\}$ are all codewords of $\C$.
The pseudo-monomials $(1-x_i)(1-x_j)$, $x_i(1-x_j)$, $x_j(1-x_i)$\ox and $x_ix_j$ respectively do not vanish on these codewords, and hence do not lie in $\CF(J_\C)$. 
Since these are the only degree two pseudo-monomials involving $x_i$ and $x_j$, we conclude that no pseudo-monomial in $\CF(J_\C)$ has variables indexed by $i$ and $j$. Hence $i$ and $j$ are adjacent in $G(\C)$, and the neighbors of $n$ form a clique as desired. 

We now argue that if $n$ is an elimination neuron with $k$ neighbors in $G(\C)$, then $n$ is a $k$-piercing of $\C$. Let $\sigma = I(n)$ and let $\tau = I(n)\cup N(n)$. Note that $\tau\backslash \sigma = N(n)$, which has size $k$.
Since $N(n)$ is a clique, for every pair of neurons in $N(n)$ there is no pseudo-monomial in $\CF(J_\C)$ whose variables are indexed by these neurons.
To show that $n$ is a $k$-piercing, we must verify two claims.

\smallskip\noindent \underline{Claim 1: $[\sigma,\tau\cup\{n\}]\subseteq \C$.}
Let $\gamma\in [\sigma,\tau\cup\{n\}]$.
It will suffice to show that every pseudo-monomial in $\CF(J_\C)$ vanishes on $\gamma$. 
First consider the pseudo-monomials that depend on $x_n$. 
Since $n$ is minimal in $P(\C)$, every pseudo-monomial in $\CF(J_\C)$ that depends on $x_n$ is of the form $x_nx_i$ with $i\notin \tau$, or $x_n(1-x_j)$ with $j\in \sigma$. 
In both cases the second term in the pseudo-monomial vanishes since $\sigma\subseteq \gamma \subseteq \tau\cup\{n\}$. 

It remains to consider pseudo-monomials that do not depend on $x_n$.
Here it suffices to consider the case where $n\notin \gamma$, since adding or removing $n$ will not affect the vanishing of the pseudo-monomial in question.
First consider those of the form $x_ix_j$. 
Since the neighbors of $n$ in $G(\C)$ form a clique on $\tau\setminus \sigma$, one of $i$ or $j$ lies outside $\tau\setminus \sigma$. 
If $i$ lies in $\sigma$, then $x_n(1-x_i)$ lies in $\CF(J_\C)$, and hence \[
x_n(x_ix_j) + (x_n(1-x_i))x_j = x_nx_j
\]
lies in $J_\C$. 
This implies that $j\notin \tau$ by (\ref{eq:CF-pierced}), and in particular $j\notin \gamma$.
Hence $x_ix_j$ vanishes on $\gamma$. 

Finally, consider a pseudo-monomial in $\CF(J_\C)$ of the form $x_i(1-x_j)$ where $i,j\in[n-1]$.
Suppose for contradiction that this pseudo-monomial does not vanish on $\gamma$.
Then $i\in \gamma$ and $j\notin \gamma$. 
In particular, $i\in \tau$ and $j\notin \sigma$. 
We will argue that $i$ and $j$ both lie in $\tau\setminus \sigma$, contradicting the fact that $G(\C)$ is a clique on these neurons. 
It suffices to show that $i\notin \sigma$ and $j\in \tau$.
Note that if $i\in\sigma$ then $x_n(1-x_i)\in J_\C$, and so
\[
x_n(1-x_i)(1-x_j) + x_nx_i (1-x_j) = x_n(1-x_j)
\]
lies in $\CF(J_\C)$, contradicting the fact that $j\notin \sigma$. 
Similarly, if $j\notin \tau$ then $x_nx_j$ lies in $J_\C$, and so
\[
x_nx_ix_j + x_nx_i(1-x_j) = x_nx_i
\]
lies in $\CF(J_\C)$ contradicting the fact that $i\in \tau$.
Hence $i$ and $j$ both lie in $\tau\setminus \sigma$, contradicting the fact that $G(\C)$ is a clique on this set. We conclude that $[\sigma,\tau\cup\{n\}]\subseteq \C$. 

\smallskip\noindent \underline{Claim 2: If $c\in \C$ and $n\in c$, then $c\in [\sigma\cup\{n\},\tau\cup\{n\}]$.}
Here we use the fact that every pseudo-monomial in $\CF(J_\C)$ vanishes on $c$, together with (\ref{eq:CF-pierced}).
For all $j\in \sigma$, the pseudo-monomial $x_n(1-x_j)$ vanishes on $c$, so we must have $\sigma\subseteq c$.
Likewise, for every $j\notin \tau\cup\{n\}$ the monomial $x_nx_j$ vanishes on $c$, and so $c\subseteq \tau\cup\{n\}$.
 This proves the result. 
 \end{proof}

\begin{lemma}\label{lem:neighborhoods}
Let $\C$ be a degree two code, and let $i,j\in[n]$ be such that $j<i$ in $P(\C)$.
Then 
\[
N(j) \subseteq I(i) \cup N(i).
\]
\end{lemma}
\begin{proof}
Suppose for contradiction that $j$ is adjacent to some neuron $\ell$ that lies in neither $I(i)$ nor $N(i)$.
Since $\ell$ does not lie in $N(i)$, there must be a pseudo-monomial $f\in \CF(J_\C)$ that depends on $x_i$ and $x_\ell$.
Since $\ell$ does not lie in $I(i)$, the only possibilities are $f = x_ix_\ell$ or $f = x_i(1-x_\ell)$.
But $j<i$, and so $x_j(1-x_i)$ lies in $\CF(J_\C)$.
Multiplying $x_j(1-x_i)$ by either $x_\ell$ or $1-x_\ell$ and adding it to $x_jf$, we obtain a pseudo-monomial in $\CF(J_\C)$ that depends on $j$ and $\ell$, contradicting the fact that $j$ and $\ell$ are adjacent in $G(\C)$.
\end{proof}

We are now ready to prove Theorem \ref{thm:characterization}, and establish our characterization of inductively pierced codes.

\characterizationthm*
\begin{proof} If $\C$ is $k$-inductively pierced, then our previous work immediately indicates the required conditions hold.
Indeed, Proposition \ref{prop:CFdegreetwo} guarantees that $\C$ is degree two. Furthermore, Proposition \ref{prop:elimination-neuron} shows that a $k$-piercing order corresponds (in reverse) to a perfect elimination order on the vertices of $G(\C)$ in which each simplicial vertex has at most $k$ neighbors, implying that $G(\C)$ is chordal with maximum clique size at most $k+1$. 

It remains to argue the converse: if $\C$ is degree two and $G(\C)$ is chordal with all cliques having size at most $k+1$, we must argue that $\C$ is $k$-inductively pierced. 
Using  Proposition \ref{prop:elimination-neuron}, we will find an elimination neuron $i\in[n]$ with at most $k$ neighbors in $G(\C)$, and prove that $\C\backslash i$ is still degree two with $G(\C\setminus i)$ chordal and all cliques having size at most $k+1$. This allows us to inductively obtain a $k$-piercing order for $\C$. 

Since $G(\C)$ is chordal, $G(\C)$ contains at least one simplicial vertex.
Choose such a vertex $i\in[n]$ that is minimal amongst all simplicial vertices in $P(\C)$---that is, there is no {\it simplicial} vertex $j<i$.
We will show that in fact $i$ is minimal in $P(\C)$; that is, $I(i) =\emptyset$.

Let $G'$ be the induced subgraph of $G(\C)$ on the vertex set $N(i)\cup I(i)$.
Since every $j\in I(i)$ is not adjacent to $i$ in $G(\C)$, it will suffice to show that $G'$ is a clique. 
For contradiction, suppose that $G'$ is not a clique.
Since $G(\C)$ is chordal, any induced subgraph is chordal.
Every chordal graph that is not a clique contains two non-adjacent simplicial vertices\footnote{This result is well-known in the chordal graph literature; see for example the proof of Theorem 8.3.27 in \cite{West}.} and hence there is a vertex $j$ that is simplicial in $G'$ and not adjacent to $i$. 

We must have $j<i$, and by Lemma \ref{lem:neighborhoods} we have $N(j) \subseteq N(i)\cup I(i)$. 
In particular, the neighborhood of $j$ in $G'$ is the same as its neighborhood in $G$.
But since this neighborhood is a clique, this contradicts our choice of $i$ as a minimal simplicial vertex. 

Hence we can find a neuron $i\in[n]$ that is simplicial in $G(\C)$ and minimal in $P(\C)$.
Up to permutation, we can assume $i = n$.
Since the largest clique in $G(\C)$ is size $k+1$, we conclude that $n$ has $k'\le k$ neighbors in $G(\C)$.
By Proposition \ref{prop:elimination-neuron}, $n$ is a $k'$-piercing of $\C$.
By (\ref{eq:CF-pierced}) we see that $\C\setminus n$ is degree two, that $G(\C\setminus n)$ is obtained from $G(\C)$ by deleting $n$, and that $P(\C\setminus n)$ is obtained from $P(\C)$ by deleting $n$. 
Hence we may conclude inductively that $\C\setminus n$ is $k$-inductively pierced, and the result follows.
\end{proof} 

We conclude with a proof of Proposition \ref{prop:greedy}, establishing that piercing orders may be constructed greedily.

\greedyprop*
\begin{proof}
It suffices to consider the case $i=n$. 
The reverse implication is clear, and it remains to argue that if $\C$ is inductively pierced then so is $\C\setminus n$. 
By Theorem \ref{thm:characterization}, it suffices to show that $\C\setminus n$ is degree two, and that $G(\C\setminus n)$ is chordal.
By (\ref{eq:CF-pierced}), $\CF(J_{\C\setminus n})\subseteq \CF(J_\C)$, and so $\C\setminus n$ is degree two.
Moreover, $G(\C\setminus n)$ is the induced subgraph of $G(\C)$ on the vertex set $[n-1]$.
Since $G(\C)$ is chordal so is $G(\C\setminus n)$, and the result follows. 
\end{proof}

\section{Realizing Inductively Pierced Codes}\label{sec:realization}

Below, we will make use of existing results regarding \emph{set-theoretically independent} collections of open balls in $\R^d$---in our language, a collection $\U = \{U_1, U_2, \ldots, U_n\}$ of balls is set-theoretically independent whenever $\code(\U)$ is the full power set $2^{[n]}$. 
For concision, we will use the simpler term \emph{independent}.
We will need the following facts regarding collections of independent balls:
\begin{itemize}
\item[(i)] The largest size of a collection of independent balls in $\R^d$ is $d+1$. This was first established by R\'enyi, R\'enyi\ox and Sur\'anyi \cite{rrs_independent}.
\item[(ii)] Let $\{U_1,\ldots, U_k\}$ be a collection of $k\le d$ independent balls in $\R^d$. 
Then the intersection of the boundaries of the $U_i$ is a $(d-k)$-dimensional sphere $S$. 
In particular, the $U_i$ are a well-formed collection of balls. 
This result appears as Theorem 1 in a paper by Anusiak~\cite{anusiak}.
\item[(iii)] In the situation of item (ii), $S$ is the set of common limit points of all atoms in the realization $\U$ (Theorem 3 of \cite{anusiak}).
\item[(iv)] In the situation of item (ii), suppose additionally that $k=d$, and let $U_{d+1}$ be any ball.
Then $S = \{p,q\}$ is a 0-dimensional sphere, and the collection of balls $\{U_1,\ldots, U_d, U_{d+1}\}$ is independent if and only if the closure of $U_{d+1}$ contains exactly one of $p$ or $q$ (Theorem 2 of \cite{anusiak}).
\end{itemize}

To build well-formed realizations of inductively pierced codes, we will use the following geometric analog of an abstract piercing. 

\begin{definition}\label{def:geom_piercing}
Let $\U = \{U_1, \ldots, U_n\}$ be a realization of $\C\subseteq 2^{[n]}$ in $\R^d$.
An interval $[\sigma,\tau]\subseteq \C$ is called \emph{pierceable} at a point $p\in \R^d$ if the codewords which arise in any small neighborhood of $p$ are exactly those in $[\sigma,\tau]$.  
Equivalently, $[\sigma,\tau]$ is pierceable at $p$ if and only if $p$ is a limit point of the atom of every codeword in $[\sigma,\tau]$ and no others.
\end{definition}

Given a collection $\{U_1,\ldots, U_n\}$ of well-formed (but not necessarily independent) balls in $\R^d$, and a point $p\in \R^d$, it will be convenient to assign $p$ to a vector $\rho(p)$ of length $n$ on the symbols $\{+,0,-\}$, as follows. 
The $i$-th coordinate of $\rho(p)$ will record whether $p$ lies in the interior, boundary\ox or complement of the closure of $U_i$, with $+$, $0$\ox and $-$ denoting these respective possibilities.
With this notation, an interval $[\sigma,\tau]$ is pierceable at $p$ if and only if $\rho(p)$ has $+$'s in the coordinates of $\sigma$, $0$'s in the coordinates of $\tau\setminus\sigma$, and $-$'s everywhere else (note that to establish this we need to use item (iii) above on the collection of balls indexed by $\tau\setminus\sigma$). 
With this observation, we can prove the following technical lemma.

\begin{lemma}\label{lem:independent}
Let $1\le k \le d$, and let $\{U_1, \ldots, U_k\}$ be a collection of independent balls in $\R^d$. 
Then every interval $[\sigma,\tau]\subseteq 2^{[k]}$ is pierceable at some point.
\end{lemma}
\begin{proof}
Let $p$ be a point in the intersection of the boundaries of all $U_i$, noting that such a point exists because $k\le d$. 
Observe that $\rho(p)$ is the all zeroes vector. 
We wish to find a point whose vector has $+$ for each index in $\sigma$, $0$ for each index in $\tau\setminus\sigma$, and $-$ for all other indices.
To achieve this, for every $i = 1, 2, \ldots, k$, we will replace $p$ by a new point, so that our new point has the correct symbols in the first $i$ indices, and $0$'s in the remaining indices.

Let $i\in [k]$, and assume that $\rho(p)$ has the correct symbols in indices less than $i$, and $0$'s for indices equal to $i$ and larger. Then $p$ lies in the intersection of the boundaries of all $U_j$ with $j > i$, as well as the $U_j$ with $j<i$ and $\rho(p)$ equal to zero in the $j$-th index. 
Let $S$ denote the intersection of these boundaries (setting $S=\R^d$ if there are no boundaries being intersected).
Observe that $\partial U_i\cap S$ is a sphere with codimension one in $S$, and that it contains $p$. 
Thus $\partial U_i$ divides $S$ into two nonempty halves, one contained in $U_i$ and the other contained in the complement of the closure of $U_i$, and $p$ is a limit point of each of these halves. 
By moving $p$ slightly inside $S$, we do not change $\rho(p)$ except possibly at the $i$-th index.
If $i\in\sigma$, we move $p$ so that it lies in the interior of $U_i$, and if $i\in[k]\setminus\tau$ we move $p$ so that it lies in the complement of the closure of $U_i$.
Otherwise, we do not move $p$.
This yields a point whose first $i$ indices have the correct symbols, and by performing this operation for all $i\in[k]$ we obtain a point with the desired $\rho(p)$, proving the result.
\end{proof}

The following lemma allows us to perform a (geometric) $k$-piercing in $\R^d$ so that all previously pierceable intervals remain pierceable, and all new intervals are also pierceable, provided that $k < d$. 
By applying this lemma repeatedly, we will prove Lemma \ref{lem:pierceable}, establishing a slightly stronger version of part (I) of Theorem \ref{thm:realization}. 

\begin{lemma}\label{lem:step}
Let $1 \le k < d$. 
Let $\C\subseteq 2^{[n]}$ be a code, suppose that the neuron $n$ is a $k$-piercing of $\C$ with associated interval $[\sigma, \tau]$, and let $\D = \C\setminus n$.
If $\D$ has a well-formed realization $\{U_1,\ldots, U_{n-1}\}$ by open balls in $\R^{d}$, and $p$ is a point where $[\sigma,\tau]$ is pierceable in this realization, then for any sufficiently small open ball $U_n$ centered at $p$ the following hold:
\begin{itemize}
    \item[(i)] $\{U_1,\ldots, U_n\}$ is a well-formed realization of $\C$ by open balls, 
    \item[(ii)] every interval in $\D$ that was pierceable in the original realization is also pierceable in this new realization, and
    \item[(iii)]  every interval in $\C$ whose top element contains $n$ is pierceable in this new realization. 
\end{itemize}
\end{lemma}
\begin{proof}
Item (i) is immediate since $[\sigma,\tau]$ is pierceable at $p$. 
The point $p$ must lie on the common intersection of the boundaries of all $U_i$ with $i\in \tau\setminus \sigma$.
This common intersection is a sphere $S$ of dimension at least one since $k<d$. 
The point $p$ lies in the interior of all $U_j$ with $j\in\sigma$, and in the complement of the closure of all $U_j$ with $j\in [n]\setminus \tau$. 
The same is true for all points in a small neighborhood of $p$, and hence $S$ contains an infinite set of points at which $[\sigma,\tau]$ is pierceable. 
For every interval $[\sigma',\tau']$ in $\D$ that was pierceable in our original realization (including the interval $[\sigma,\tau]$ itself), choose a point $p'\neq p$ at which this interval was pierceable.
Choosing $U_n$ small enough that its closure avoids all such $p'$ guarantees item (ii).

For item (iii), define $\gamma = (\tau\cup\{n\})\setminus \sigma$ and observe that the $U_i$ with $i$ in $\gamma$ comprise a collection of $k+1\le d$ independent balls.
Each interval $[\sigma',\tau']$ in $\C$ with $n\in\tau'$ must be a subset of $[\sigma, \tau\cup \{n\}]$. 
Up to removing the indices in $\sigma$, such an interval corresponds to an interval in $2^\gamma$, which is pierceable by Lemma \ref{lem:independent}. 
Since $n\in \tau$, the point where this interval is pierceable lies in the closure of $U_n$.
In particular, this point lies in the interior of all $U_i$ with $i\in \sigma$, and in the complement of the closure of all $U_j$ with $j\in [n-1]\setminus \tau$. 
Hence the interval $[\sigma',\tau']$ is pierceable at this point in our new realization, proving the result. 
\end{proof}

\begin{lemma}\label{lem:pierceable}
Let $\C\subseteq 2^{[n]}$ be a $k$-inductively pierced code. 
Then $\C$ has a well-formed realization by open balls in $\R^{k+1}$, with the following property: every interval of rank at most $k$ in $\C$ is pierceable at some point.
\end{lemma}
\begin{proof}
We may assume that the natural linear order on $n$ is a piercing order, and proceed by induction on $n$, applying Lemma \ref{lem:step} at each step. 
The base case $n=1$ is immediate by choosing any open ball $U_1$ in $\R^{k+1}$. 
For the inductive step, one must only note that the intervals described in items (ii) and (iii) account for all intervals of rank at most $k$ in $\C$.
Indeed, when performing a $k$-piercing, the only new intervals of rank at most $k$ that are introduced to the code are those of the form $[\sigma', \tau'\cup \{n\}]$, which are covered by case (iii), while case (ii) covers all previously introduced intervals of rank at most $k$. 
\end{proof}

We now turn our attention to the task of performing geometric $k$-piercings in $\R^k$.
Since the intersection of $k$-many independent spheres in $\R^k$ is a 0-dimensional sphere consisting of only two points, not every $k$-piercing can be realized geometrically in $\R^k$.
However, when a code is splittable (recall Definition \ref{def:splittable}), certain $k$-piercings can be achieved.
We say that an interval $[\sigma,\tau]$ in a splittable code $\C$ is \emph{accessible} if it has rank less than $k$, or if it has rank exactly $k$ and the following holds: the intersection of $\sigma$ with the attaching set $A\sqcup B$ of $\tau\setminus \sigma$ is equal to $A$ or $B$.

\begin{lemma}\label{lem:splittable-realization}
Let $k\ge 1$, let $\C\subseteq 2^{[n]}$ be an inductively $k$-pierced code, and suppose that $\C$ is splittable. 
Then $\C$ has a well-formed realization by open balls in $\R^k$, with the following property: every accessible interval in $\C$ is pierceable at some point.
\end{lemma}
\begin{proof}
By relabeling, we can reduce to the case where the natural linear order on $[n]$ is a piercing order. 
Similarly to the proof of Lemma \ref{lem:pierceable}, we will proceed by induction on $n$.
For the base case $n=1$, any open ball $U_1$ in $\R^k$ has every interval pierceable, and the result follows. For $n\ge 2$, define $\D \eqdef \C\setminus n$, let $[\sigma, \tau]$ be the interval associated to the $k$-piercing of $\C$ by $n$, and let $\{U_1,\ldots, U_{n-1}\}$ be a realization of $\D$ satisfying the conclusion of the lemma. 
If $[\sigma, \tau]$ has rank less than $k$ (i.e. $n$ is a $k'$-piercing of $\C$ with $k'<k$) then Lemma \ref{lem:step} yields a realization of $\C$ with the desired properties.

This leaves the case that $[\sigma, \tau]$ has rank exactly $k$. 
In this case, let $A\sqcup B$ be the attaching set of $\tau\setminus \sigma$ in $\C$, and assume without loss of generality that $n\in A$. 
Since $\C$ is splittable and $n$ is minimal in $P(\C)$, $n$ is less than every $a\in A$ and is not comparable to any $b\in B$.
It follows that $\sigma$ contains $A$ and no elements of $B$. 
Noting that the attaching set of $\tau\setminus \sigma$ in $\D$ is $A\setminus \{n\} \sqcup B$, we see that the interval $[\sigma,\tau]$ is accessible in $\D$, and by inductive hypothesis there exists a point $p$ where this interval is pierceable in the realization $\{U_1,\ldots, U_{n-1}\}$. 

Consider a small open ball $U_n$ centered at $p$. 
By making $U_n$ small enough, we can guarantee that $\{U_1,\ldots, U_n\}$ is a well-formed realization of $\C$ by open balls, and that every interval in $\D$ that was pierceable remains pierceable, except possibly the interval $[\sigma,\tau]$ itself. 
Moreover, Lemma \ref{lem:independent} guarantees that any new intervals $[\sigma', \tau']$ of rank $k$ or less with $n\in \tau'$ are pierceable in our new realization as well. 
The only remaining interval in $\C$ which could be accessible is $[\sigma, \tau]$. 
Observe that this interval is accessible in $\C$ if and only if $B = \emptyset$. 
In this case, the 0-dimensional sphere which is the intersection of the boundaries of all $U_i$ indexed by $\tau\setminus \sigma$ will consist of $p$ and a point $q$, and the interval $[\sigma, \tau]$ will be pierceable at both these points in the realization of $\D$.
The set $U_n$ must avoid the point $q$, lest the spheres indexed by $(\tau\setminus\sigma)\cup \{n\}$ not form an independent collection, and so $[\sigma,\tau]$ remains pierceable at $q$ in our constructed realization of $\C$.
This proves the result. 
\end{proof}

The above lemmas give the required constructive results for Theorem \ref{thm:realization}, and it remains to establish the theorem's restrictions on well-formed realizations of inductively pierced codes---namely that if $\C$ has a well-formed realization by open balls in $\R^k$ then $\C$ must be splittable, and that realizations in $\R^{k-1}$ are not possible. 

\begin{lemma}\label{lem:splittable}
Let $\C\subseteq 2^{[n]}$ be an inductively pierced code, suppose that the largest clique in $G(\C)$ has size $k+1$, and suppose that $\C$ has a well-formed realization by open balls in $\R^k$.
Then $\C$ is splittable. 
\end{lemma}
\begin{proof}
Let $\sigma$ be the set of vertices of a clique of size $k$ in $G(\C)$.
In any well-formed realization of $\C$ in $\R^k$ by balls, the intersection of the boundaries of the balls indexed by $\sigma$ will be a $0$-dimensional sphere consisting of two points, $p$ and $q$.
The balls indexed by the attaching set of $\sigma$ must each contain exactly one of $p$ or $q$ in their closure. 
Let $A$ be the set of neurons whose associated open balls contain $p$, and let $B$ be the neurons whose associated balls contain $q$. 
Clearly $S = A\sqcup B$. 

It remains to argue that $P(\C)$ is a total order on $A$ and $B$ respectively, and that no element of $A$ is comparable to any element of $B$.
To see that $P(\C)$ is a total order on $A$, first recall that the attaching set of $\sigma$ is an independent set in $G(\C)$. Hence if $i,j\in A$, then $\CF(J_\C)$ contains one of the pseudo-monomials $x_i(1-x_j)$, $x_j(1-x_i)$, or $x_ix_j$. 
Since $U_i$ and $U_j$ share the point $p$, the pseudo-monomial $x_ix_j$ does not occur in $\CF(J_\C)$, and we conclude that $i$ and $j$ are comparable in $P(\C)$.
Since any pair of elements in $A$ are comparable, $P(\C)$ is a total order on $A$.
A symmetric argument establishes that $P(\C)$ is a total order on $B$.
Finally, note that if $i\in A$ and $j\in B$, then $i$ and $j$ are not comparable in $P(\C)$ because $U_i$ contains $p$ while $U_j$ does not, and symmetrically $U_j$ contains $q$ while $U_i$ does not.
\end{proof}

\begin{lemma}\label{lem:clique_lower_bound}
Let $\C\subseteq 2^{[n]}$ be an inductively pierced code, and suppose $G(\C)$ contains a clique of size $k+1$.
Then $\C$ does not have a well-formed realization by open balls in $\R^{k-1}$. 
\end{lemma}
\begin{proof}
By relabeling we may assume that $G(\C)$ has a clique on $[k+1]$.
The fact that $\C$ is inductively pierced implies that the restriction of $\C$ to $[k+1]$ is the full power set $2^{[k+1]}$.
Indeed, up to ignoring the indices outside of $[k+1]$, any piercing order successively builds sub-intervals of $2^{[k+1]}$ with ranks $0,1,\ldots, k+1$, each of which are contained in $\C$. 
We conclude that if $\C$ had a well-formed realization by open balls in $\R^{k-1}$, then the balls indexed by $[k+1]$ would form an independent collection.
But the largest collection of independent balls in $\R^{k-1}$ has size $k$.
This proves the result.
\end{proof}

\realizationthm*
\begin{proof}
Lemma \ref{lem:pierceable} established a realization of the desired form for (I).
Lemmas \ref{lem:splittable-realization} and \ref{lem:splittable} established the reverse and forward implications needed for (II) respectively.
Finally, Lemma \ref{lem:clique_lower_bound} established (III).
\end{proof}

\section{Computing Piercing Orders and Concluding Remarks}\label{sec:conclusion}

Motivated by the work of Gross, Obatake\ox and Youngs \cite{visualization}, we provide a complete characterization of inductively pierced codes. Additionally, we show that inductively pierced codes can be realized with open balls, and we exactly characterize the smallest dimension in which such realizations exist.
A consequence of our characterization is that an inductively pierced code $\C\subseteq 2^{[n]}$ can be realized in dimension $n-1$ or less.
This strongly contrasts the general situation for convex codes: Jeffs \cite{embedding_phenomena} showed that there are codes $\C\subseteq 2^{[n]}$ which can only be realized in dimensions that are \emph{exponential} in terms of $n$. 

Within our results and their proofs, we have developed an implicit algorithm for a) recognizing if a code is inductively pierced, b) computing a piercing order, c) determining the minimal dimension in which the code may be realized with open balls,  and d) producing such a realization.  
We summarize this algorithm below.
Recall that the codes we work with satisfy the mild requirements (1)--(3) outlined in the introduction.\\

\noindent {\bf Algorithm for computing a piercing order.}  
\begin{enumerate}
\item Compute $\CF(J_\C)$.\\
\emph{Note:} Current algorithms for this are not necessarily efficient---see for example the discussion by Garcia, Puente, Kruse, Liu, Miyata, Petersen, Phillipson\ox and Shiu \cite[Page 1]{grobner}. 
We conjecture below that $\CF(J_\C)$ can be computed efficiently when $\C$ is inductively pierced, even if we do not know this a priori. 
 \item Verify that $\C$ is degree two. If not, then  $\C$ is not inductively pierced. 
\item Form $G(\C)$ and $P(\C)$ from $\CF(J_\C)$.
\item Search for an elimination neuron $i\in [n]$ (recall Definition \ref{def:elimination-neuron}). 
This can be done in $O(n^2)$ time by computing the neighborhood of each neuron in $G(\C)$ and checking whether it is simplicial, and then comparing all simplicial vertices according to $P(\C)$ to find a minimal one. 
If no elimination neuron exists, then $\C$ is not inductively pierced. 
\item Perform a permutation of $[n]$ so that $n$ is our elimination neuron.
Record this permutation and the clique consisting of $n$ and its neighbors. 
\item Delete $n$ from $G(\C)$ and $P(\C)$ to obtain $G(\C\setminus n)$ and $P(\C\setminus n)$ respectively. 
Repeat from Step 4 until no neurons remain, tracking the permutations from each time that we perform in Step 5. 
\item Compose the recorded permutations to obtain a piercing order for $\C$.
The smallest value of $k$ for which $\C$ is $k$-inductively pierced is size of the largest clique recorded in an instance of Step 5, minus one. 
\item Determine if $\C$ is splittable.
Each clique of size $k$ will be contained in one of the $(k+1)$-cliques recorded in Step 5, so there are at most $n(k+1)$ such cliques to check. 
For each, we can compute its attaching set and check whether it is possible to partition it as in Definition \ref{def:splittable}.
\item If $\C$ is splittable, we form a realization in $\R^k$. Otherwise, we form a realization in $\R^{k+1}$.
To realize $\C$ by open balls, sequentially add open balls according to the piercing order.
Lemmas \ref{lem:step}, \ref{lem:pierceable}\ox and \ref{lem:splittable-realization} determine the points on the boundaries of these balls where subsequent balls should be added to form the appropriate code.
\end{enumerate}

\begin{conjecture}
Fix $d\ge 2$.
Given a code $\C\subseteq 2^{[n]}$, there is a polynomial-time algorithm to determine whether or not every pseudo-monomial in $\CF(J_\C)$ has degree at most $d$. 
Here, the input to the algorithm is a list of the codewords in $\C$, with each codeword given as a binary vector of length $n$.
\end{conjecture}

Note that when $\CF(J_\C)$ consists of pseudo-monomials with degree bounded by a fixed constant $d$, there is  a polynomial-time algorithm for determining $\CF(J_\C)$.
Indeed, the total number of pseudo-monomials in $\mathbb F_2[x_1,\ldots, x_n]$ of degree at most $d$ is $2^d\binom{n}{d} + 2^{d-1}\binom{n}{d-1} + \cdots + 2d + 1 = O(n^d)$, and so we could simply evaluate each of these pseudo-monomials on all codewords to determine which ones lie in $J_\C$, and the compare them pairwise to find the minimal ones, which will comprise $\CF(J_\C)$.
Hence if this conjecture is true, then our results imply that there is a polynomial-time algorithm for determining whether or not a code is inductively pierced, even if we are not initially given $\CF(J_\C)$. 

\paragraph{Concluding Remarks.} 

Our work proves that inductively pierced codes must be realizable using open balls, but there are codes that are not inductively pierced but nevertheless admit a realization by open balls. Indeed, there are such codes that are not degree two, and there are also such codes that are degree two but have $G(\C)$ non-chordal. 
The code $\C = \{12, 13, 23, 1,2,3,\emptyset\}$ is a simple example of the first case, as its canonical form is $\{x_1x_2x_3\}$, but it can be realized in two dimensions using open balls, as in Figure~\ref{fig:not-pierced}. 
For an example of the second case, consider $\mathcal D = \{12, 14, 23, 34, 1,2,3,4,\emptyset\}$, which has canonical form $\{x_1x_3, x_2x_4\}$ and hence its associated graph is a 4-cycle, which is not chordal. However, $\D$ has a realization by open balls in $\R^2$ as shown in Figure~\ref{fig:not-pierced}. 

\begin{figure}[h]
\[
\includegraphics{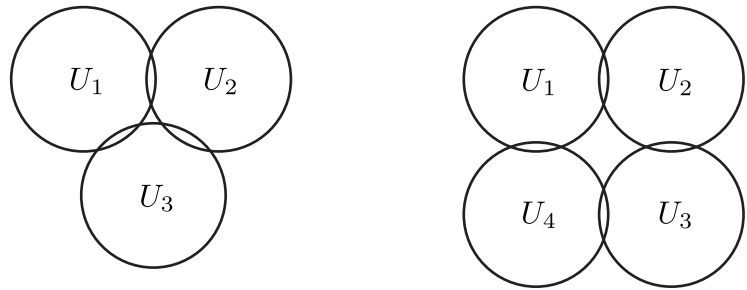}
\]
\caption{Two codes that are not inductively pierced, but can be realized with well-formed collections of open balls in the plane. (Left) The code $\C$ is not degree two. (Right) The graph $G(\D)$ is not chordal. }\label{fig:not-pierced}
\end{figure}

 It seems plausible to seek larger classes of codes that admit realizations that are constructed iteratively from open balls. 
 However, codes such as $\D$ above, whose realizations cover non-contractible subsets of $\R^d$, will pose a challenge to such an approach. Had we arranged $U_1, U_2$\ox and $U_3$ in a vertical line, then there would have been no way to add an appropriate open ball (or any convex set) $U_4$ at the final step.
 This highlights the special structure of inductively pierced codes, whose realizations do not create loops or other topologically interesting features. \\

\begin{figure}[h]
\[
\includegraphics{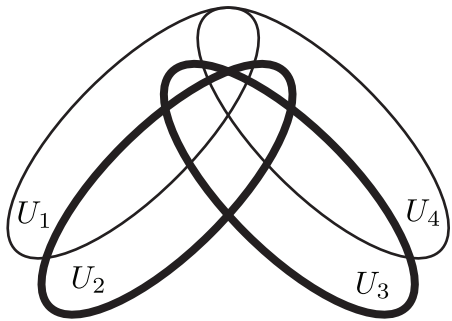}
\]
\caption{An inductively pierced code realized by ellipsoidal regions in the plane. This code cannot be realized by open balls in the plane.}\label{fig:ellipses}
\end{figure}

Finally, note that when realizations are not restricted to open balls, our characterization of realization dimension does not apply, even when a code is inductively pierced.
For example, consider the complete code on four neurons, $\C = 2^{[4]}$. This code is inductively pierced---its canonical form is empty and thus trivially meets the degree two condition, and its graph $G(\C)$ is the complete graph $K_4$ and thus is chordal. 
Theorem \ref{thm:realization} implies that $\C$ is not realizable in $\R^2$ using open balls, but it does have realization in $\R^2$ using ellipsoidal regions, as shown in Figure~\ref{fig:ellipses}. 
The principal difference between this realization and any realization using open balls is that the intersection of two ellipses need not be a lower-dimensional ellipse. In the figure, the sets $U_2$ and $U_3$ ``pass through" one another, and their boundaries share four common points. 
Accounting for such possibilities renders it more difficult to formulate simple inductive constructions for realizations by ellipses. A characterization of minimal dimension for realizations of inductively pierced codes using ellipses (or other types of convex regions) remains an open question.

\bibliographystyle{plain}
\bibliography{inductively_pierced.bib}

\begin{thebibliography}{10}

\bibitem{anusiak}
J.~Anusiak.
\newblock On set-theoretically independent collections of balls.
\newblock {\em Colloquium Mathematicum}, 13(2):223--233, 1964.

\bibitem{rigid_structures}
Patrick Chan, Katherine Johnston, Joseph Lent, Alexander Ruys~de Perez, and
  Anne Shiu.
\newblock Nondegenerate neural codes and obstructions to closed-convexity.
\newblock {\em arXiv e-prints: 2011.04565}, 2020.

\bibitem{undecidability}
Aaron Chen, Florian Frick, and Anne Shiu.
\newblock Neural codes, decidability, and a new local obstruction to convexity.
\newblock {\em SIAM Journal on Applied Algebra and Geometry}, 3(1):44--66,
  2019.

\bibitem{cgik}
Joshua Cruz, Chad Giusti, Vladimir Itskov, and Bill Kronholm.
\newblock On open and closed convex codes.
\newblock {\em Discrete \& Computational Geometry}, 61:247--270, 2016.

\bibitem{local_obstructions}
Carina Curto, Elizabeth Gross, Jack Jeffries, Katherine Morrison, Mohamed Omar,
  Zvi Rosen, Anne Shiu, and Nora Youngs.
\newblock What makes a neural code convex?
\newblock {\em SIAM Journal on Applied Algebra and Geometry}, 1(1):222--238,
  2017.

\bibitem{civy}
Carina Curto, Vladimir Itskov, Alan Veliz-Cuba, and Nora Youngs.
\newblock The neural ring: an algebraic tool for analyzing the intrinsic
  structure of neural codes.
\newblock {\em Bulletin of Mathematical Biology}, 75(9):1571--1611, 2013.

\bibitem{non_monotone}
Brianna {Gambacini}, R.~Amzi {Jeffs}, Sam {Macdonald}, and Anne {Shiu}.
\newblock {Non-monotonicity of closed convexity in neural codes}.
\newblock {\em Vietnam Journal of Mathematics}, 2021.

\bibitem{grobner}
Rebecca Garcia, Luis Garcia-Puente, Ryan Kruse, Jessica Liu, Dane Miyata, Ethan
  Petersen, Kaitlyn Phillipson, and Anne Shiu.
\newblock Gr\"obner bases of neural ideals.
\newblock {\em International Journal of Algebra and Computation},
  28(4):553--571, 2018.

\bibitem{visualization}
Elizabeth Gross, Nida~Kazi Obatake, and Nora Youngs.
\newblock Neural ideals and stimulus space visualization.
\newblock {\em Advances in Applied Mathematics}, 95:65--95, 2018.

\bibitem{all_embedding_vectors}
R.~Amzi Jeffs.
\newblock Open, closed, and non-degenerate embedding dimensions of neural
  codes.
\newblock {\em arXiv e-prints: 2111.01010}, 2021.

\bibitem{embedding_phenomena}
R.~Amzi Jeffs.
\newblock Embedding dimension phenomena in intersection complete codes.
\newblock {\em Selecta Math. (N.S.)}, 28(1):Paper No. 18, 37, 2022.

\bibitem{order_forcing}
R.~Amzi Jeffs, Caitlin Lienkaemper, and Nora Youngs.
\newblock Order-forcing in neural codes.
\newblock {\em arXiv e-prints: 2008.13192}, 2020.

\bibitem{kunin_lienkaemper_rosen}
Alexander {Kunin}, Caitlin {Lienkaemper}, and Zvi {Rosen}.
\newblock Oriented matroids and combinatorial neural codes.
\newblock {\em arXiv e-prints: 2002.03542}, page arXiv:2002.03542, 2020.

\bibitem{non_local_obstructions}
Caitlin Lienkaemper, Anne Shiu, and Zev Woodstock.
\newblock Obstructions to convexity in neural codes.
\newblock {\em Advances in Applied Mathematics}, 85:31--59, 2017.

\bibitem{miller_sturmfels}
Ezra Miller and Bernd Sturmfels.
\newblock {\em Combinatorial commutative algebra}, volume 227 of {\em Graduate
  Texts in Mathematics}.
\newblock Springer-Verlag, New York, 2005.

\bibitem{rrs_independent}
A.~R\'{e}nyi, C.~R\'{e}nyi, and J.~Sur\'{a}nyi.
\newblock Sur l'ind\'{e}pendance des domaines simples dans l'espace {E}uclidien
  \`a {$n$} dimensions.
\newblock {\em Colloq. Math.}, 2:130--135, 1951.

\bibitem{stanley}
Richard~P. Stanley.
\newblock {\em Combinatorics and commutative algebra}, volume~41 of {\em
  Progress in Mathematics}.
\newblock Birkh\"{a}user Boston, Inc., Boston, MA, second edition, 1996.

\bibitem{West}
Douglas~B. West.
\newblock {\em Combinatorial Mathematics}.
\newblock Cambridge University Press, 2021.

\end{thebibliography}

\end{document}